\documentclass[11pt]{amsart}
\usepackage{amsmath,amssymb,amsthm}
\usepackage[margin=1.15in]{geometry}
\usepackage{enumitem}
\usepackage{xcolor}
\usepackage[colorlinks=true,linkcolor=blue!50!black,citecolor=blue!50!black,urlcolor=blue!50!black]{hyperref}

\newtheorem{theorem}{Theorem}[section]
\newtheorem{lemma}[theorem]{Lemma}
\newtheorem{proposition}[theorem]{Proposition}
\newtheorem{corollary}[theorem]{Corollary}
\theoremstyle{definition}
\newtheorem{definition}[theorem]{Definition}
\newtheorem{example}[theorem]{Example}
\newtheorem{problem}{Problem}
\theoremstyle{remark}
\newtheorem{remark}[theorem]{Remark}

\makeatletter
\@namedef{subjclassname@2020}{\textup{2020} Mathematics Subject Classification}
\makeatother

\newcommand{\K}{\mathbb{K}}
\newcommand{\R}{\mathbb{R}}
\newcommand{\C}{\mathbb{C}}
\newcommand{\Z}{\mathbb{Z}}
\newcommand{\N}{\mathbb{N}}
\newcommand{\Q}{\mathbb{Q}}
\newcommand{\cl}{\operatorname{cl}}

\begin{document}

\title{Semigroups of matrices with dense orbits\\ (revised version)}
\author{Mohammad Javaheri}
\address{Department of Mathematics and Statistics, Siena University\\
515 Loudon Road, Loudonville, NY 12211, USA}
\email{mjavaheri@siena.edu}
\date{Revised version, October 2026}

\begin{abstract}
For $\K=\R$ or $\C$ and every $n\ge 1$ we construct two $n\times n$ matrices
$A$ and $B$ such that the orbit of every vector $p\in\K^n$ with $p_1\ne 0$
under the semigroup generated by $A$ and $B$ is dense in $\K^n$; the
remaining vectors lie in an invariant hyperplane, so the set of vectors with
dense orbit is exactly the complement of a hyperplane. Thus two generators
always suffice for a hypercyclic semigroup of matrices, in contrast with the
commutative case, where $n+1$ generators are needed on $\C^n$. We give a
transparent reformulation of the hypotheses, a criterion for the density of
$\{a^mb^n\}$ that makes them checkable, and explicit examples with complete
proofs. As applications we obtain affine pairs on $\K^n$ all of whose orbits
are dense, and, for every $n$, a pair of real matrices with an orbit that is
somewhere dense but not dense in $\R^n$.
\end{abstract}


\maketitle

\section{Introduction}

Let $X$ be a topological vector space and $T\colon X\to X$ a continuous linear
operator. $T$ is \emph{hypercyclic} if some $x\in X$ has a dense orbit
$\{x,Tx,T^2x,\dots\}$. Answering a question of Rolewicz~\cite{Rol}, Ansari~\cite{Ans}
and Bernal-Gonz\'alez~\cite{Ber} proved that every separable
infinite-dimensional Banach space supports a hypercyclic operator. In
contrast, no operator on a finite-dimensional space is hypercyclic; this can
be seen from the Jordan normal form (see Kitai~\cite{Kit}). In finite dimension
it is therefore natural to replace the single operator by a finitely
generated semigroup.

\begin{definition}\label{def:hc}
Let $\Gamma=\langle T_1,\dots,T_k\rangle$ be the semigroup (with identity)
generated by linear maps $T_1,\dots,T_k$ of $\K^n$, $\K=\R$ or $\C$. For
$x\in\K^n$ write $\Gamma x=\{Tx: T\in\Gamma\}$ for its orbit. We say $\Gamma$ is
\emph{hypercyclic} if $\Gamma x$ is dense in $\K^n$ for some $x\in\K^n$; such
an $x$ is a \emph{hypercyclic vector}.
\end{definition}

Feldman~\cite{Fel} initiated the study of hypercyclic tuples (commuting
families) of matrices: he showed that $\C^n$ supports a hypercyclic
$(n+1)$-tuple of diagonal matrices and that no $n$-tuple of diagonalizable
matrices on $\R^n$ or $\C^n$ has a somewhere dense orbit. Costakis,
Hadjiloucas and Manoussos~\cite{CHM1} constructed hypercyclic $n$-tuples of
non-diagonalizable matrices on $\R^n$ for $n\ge2$, and they also showed
\cite{CHM2} that in every finite dimension there are commuting pairs which are
locally hypercyclic but not hypercyclic. For commuting tuples on $\C^n$ the
minimal number of generators of a hypercyclic tuple is $n+1$
\cite{Aya,CHM3,Shk}, and on $\R^n$ it is $\frac n2+\frac{5+(-1)^n}{4}$~\cite{Shk}.
So in the commutative setting the number of generators must grow with the
dimension. This raises the question of the minimal number of generators of a
hypercyclic semigroup on $\K^n$ when commutativity is dropped. Since a single
matrix is never hypercyclic, the answer is at least $2$; Theorem~\ref{thm:main}
shows that it is exactly $2$ for every $n\ge1$.

\subsection*{Notation}
For an $n\times n$ matrix $A$, $A_{ij}$ is its $(i,j)$ entry and $A_i:=A_{ii}$.
$I$ (or $I_s$) is the identity, $e_1,\dots,e_n$ the standard basis, and
$\Delta=e_1e_1^{T}$ the matrix with $\Delta_{11}=1$ and all other entries $0$.
$\N=\{0,1,2,\dots\}$. For $z\in\K^n$, $z_i$ is its $i$th coordinate, and
$\cl(Y)$ is the closure of $Y$. We write $\K^*=\K\setminus\{0\}$.

\begin{definition}\label{def:gen}
A pair $(a,b)\in\K^2$ is \emph{generating} if $\{a^mb^k: m,k\in\N\}$ is dense in
$\K$ and $|a|<|b|$. For generating pairs we write $(a,b)\prec(c,d)$ if
\[
\frac{\ln|a|}{\ln|b|}<\frac{\ln|c|}{\ln|d|}.
\]
\end{definition}

If $(a,b)$ is generating then necessarily $|a|<1<|b|$: if $|a|,|b|\le1$ or
$|a|,|b|\ge1$ the set $\{|a^mb^k|\}$ is bounded or bounded away from $0$,
and if $|b|=1$ then $\{|a^mb^k|\}=\{|a|^m\}$ is discrete in $(0,\infty)$. In
particular the quotients in the definition of $\prec$ are well defined and
negative. Lemma~\ref{lem:gen} below characterizes generating pairs.

\subsection*{Main results}

\begin{theorem}\label{thm:main}
Let $A$ and $B$ be $n\times n$ matrices over $\K$ with $A$ lower triangular
and $B$ diagonal. Suppose:
\begin{enumerate}[label=\textup{(\roman*)}]
\item the diagonal entries satisfy
\begin{equation}\label{eq:diag}
0<|B_n|<\dots<|B_2|<|B_1|<1<|A_1|<|A_2|<\dots<|A_n|,
\end{equation}
and
\begin{equation}\label{eq:chain}
\Bigl(\frac{B_n}{B_1},\frac{A_n}{A_1}\Bigr)\prec\dots\prec
\Bigl(\frac{B_2}{B_1},\frac{A_2}{A_1}\Bigr)\prec(B_1,A_1);
\end{equation}
\item every entry of the eigenvector $u$ of $A$ defined by
\begin{equation}\label{eq:u}
Au=A_1u,\qquad u_1=1,
\end{equation}
is non-zero. Equivalently, all entries of the first column of
$(A_1^{-1}A-I+\Delta)^{-1}$ are non-zero.
\end{enumerate}
Then for every $p\in\K^n$ with $p_1\ne0$ and every $N\in\N$, the set
\begin{equation}\label{eq:words}
W^{(N)}(p)=\bigl\{B^{\delta_n}A^{\gamma_n}\cdots B^{\delta_1}A^{\gamma_1}p:\ \gamma_i,\delta_i\ge N\bigr\}
\end{equation}
is dense in $\K^n$. In particular the orbit of $p$ under $\Gamma=\langle A,B\rangle$
is dense.
\end{theorem}

The vector $u$ in \eqref{eq:u} exists and is unique, since $A$ is lower
triangular with $A_i\ne A_1$ for $i\ge2$: solving row by row gives
$u_i=-(A_i-A_1)^{-1}\sum_{t<i}A_{it}u_t$. The equivalence in (ii) is proved in
Lemma~\ref{lem:J}. For instance, if the only non-zero off-diagonal entries of
$A$ lie in the first column, then $u_i=-A_{i1}/(A_i-A_1)$, so (ii) holds iff
$A_{i1}\ne0$ for all $i\ge2$. Since the $u_i$ are rational functions of the
entries of $A$, condition (ii) defines a non-empty Zariski open, hence open and
dense, subset of the lower triangular matrices with prescribed diagonal.

The hypothesis $p_1\ne0$ cannot be dropped: the hyperplane
$\{x\in\K^n:x_1=0\}$ is invariant under every lower triangular matrix. Hence:

\begin{corollary}\label{cor:main}
Under the hypotheses of Theorem~\ref{thm:main}:
\begin{enumerate}[label=\textup{(\alph*)}]
\item the set of hypercyclic vectors of $\Gamma=\langle A,B\rangle$ is exactly
$\{p\in\K^n:p_1\ne0\}$, an open dense set of full Lebesgue measure;
\item $\Gamma$ is topologically transitive: for all non-empty open $U,V\subseteq\K^n$
there is $T\in\Gamma$ with $T(U)\cap V\ne\emptyset$;
\item if $x^{(1)},\dots,x^{(k)}\in\K^n$ and $\bigcup_j\Gamma x^{(j)}$ is dense, then
$\Gamma x^{(j)}$ is dense for some $j$.
\end{enumerate}
\end{corollary}

\begin{proof}
(a) is Theorem~\ref{thm:main} together with the invariance of $\{x_1=0\}$.
(b) Pick $p\in U$ with $p_1\ne0$; some $Tp$ lies in $V$.
(c) Not all $x^{(j)}$ lie in the closed invariant proper subspace $\{x_1=0\}$,
so some $x^{(j)}_1\ne0$, and (a) applies.
\end{proof}

Part (c) answers, for this class of semigroups, the analogue of Herrero's
question on multi-hypercyclicity (see Problem~\ref{prob:multi}).

Explicit pairs satisfying the hypotheses are given in Section~\ref{sec:examples},
with full proofs (Proposition~\ref{prop:examples}). Using
Theorem~\ref{thm:main} we also obtain the following results, proved in
Section~\ref{sec:apps}.

\begin{theorem}\label{thm:affine}
Let $A$ be lower triangular and $B$ diagonal $n\times n$ matrices over $\K$
satisfying \eqref{eq:diag} and
\begin{equation}\label{eq:chain-affine}
(B_n,A_n)\prec\dots\prec(B_2,A_2)\prec(B_1,A_1).
\end{equation}
Let $v\in\K^n$ be such that every entry of $(A-I)^{-1}v$ is non-zero. Then
every orbit of the semigroup generated by the affine maps
\[
D(x)=Ax+v,\qquad E(x)=Bx
\]
is dense in $\K^n$.
\end{theorem}

\noindent In dimension one, if $(a,b)$ is generating then every orbit of the
semigroup generated by $x\mapsto ax$ and $x\mapsto bx+c$ ($c\ne0$) is dense in $\R$
(see \cite{Jav} for the classification of such pairs on $\R$), so
Theorem~\ref{thm:affine} is a higher-dimensional analogue.

\begin{theorem}\label{thm:halfspace}
Let $\K=\R$, let $A$ be lower triangular and $B$ diagonal, and assume
\eqref{eq:diag}, condition \textup{(ii)} of Theorem~\ref{thm:main}, and:
\begin{enumerate}[label=\textup{(\alph*)}]
\item $A_1>0$, $B_1>0$ and $\ln B_1/\ln A_1\notin\Q$;
\item $(B_i/B_1,A_i/A_1)$ is generating for $2\le i\le n$, and
\[
\frac{\ln|B_n/B_1|}{\ln|A_n/A_1|}<\dots<\frac{\ln|B_2/B_1|}{\ln|A_2/A_1|}<\frac{\ln B_1}{\ln A_1}.
\]
\end{enumerate}
Then for every $p\in\R^n$ with $p_1>0$ the closure of the orbit $\langle A,B\rangle p$
is the closed half-space $\{x\in\R^n:x_1\ge0\}$. In particular every such orbit
is somewhere dense but not dense, and $\langle A,B\rangle$ has no dense orbit.
\end{theorem}

Theorem~\ref{thm:halfspace} settles the real half of a problem raised in the
first version of this paper (Problem~\ref{prob:sd} below).

\section{Generating pairs}\label{sec:gen}

Let $G$ denote one of the groups $(0,\infty)$, $\R^*$, $\C^*$ under multiplication.
Each is a locally compact abelian group, and the map $z\mapsto(\ln|z|,z/|z|)$
identifies it with $\R\times C$, where $C$ is respectively $\{1\}$, $\{\pm1\}$
or the unit circle, a compact group.

\begin{lemma}\label{lem:large}
Let $a,b\in G$ with $|a|<1<|b|$, and suppose the subgroup generated by $a$ and
$b$ is dense in $G$. Then for every $N\in\N$ the set
$S_N=\{a^mb^k: m,k\ge N\}$ is dense in $G$, hence in $\cl(G)$.
\end{lemma}

\begin{proof}
Write $\alpha=\ln|a|<0$ and $\beta=\ln|b|>0$.

\emph{Step 1: for every neighborhood $U$ of $1$ in $G$ and every $M_0$ there
are $m,k\ge M_0$ with $a^mb^k\in U$.} Let $V$ be a symmetric neighborhood of
$1$ with $VV\subseteq U$. For $m\in\N$ put $k(m)=\lceil -m\alpha/\beta\rceil$
and $z_m=a^mb^{k(m)}$. Then $0\le\ln|z_m|<\beta$, so every $z_m$ lies in the
compact set $Q=\{z\in G:1\le|z|\le|b|\}$. Cover $Q$ by $r$ translates $cV$.
Fix $M\ge M_0$. Among $z_0,z_M,z_{2M},\dots,z_{rM}$ two, say $z_{iM}$ and
$z_{jM}$ with $i<j$, lie in the same translate $cV$, and then
$z_{jM}z_{iM}^{-1}\in VV\subseteq U$. This element equals $a^{m}b^{k}$ with
$m=(j-i)M\ge M$ and
\[
k=k(jM)-k(iM)\ge \frac{-(j-i)M\alpha}{\beta}-1\ge \frac{-M\alpha}{\beta}-1 .
\]
Both exponents are $\ge M_0$ once $M$ is large.

\emph{Step 2.} Let $w\in G$ and let $W$ be a neighborhood of $w$. Choose a
neighborhood $W'$ of $w$ and a neighborhood $U$ of $1$ with $W'U\subseteq W$.
By density of the group there are $m,k\in\Z$ with $a^mb^k\in W'$. By Step 1
there are $m',k'\ge N+|m|+|k|$ with $a^{m'}b^{k'}\in U$. Then
$a^{m+m'}b^{k+k'}\in W$ and $m+m',k+k'\ge N$.
\end{proof}

\begin{lemma}[Kronecker]\label{lem:kron}
A subgroup $\Lambda$ of $\R^d$ is dense iff there is no non-zero linear
functional $f\colon\R^d\to\R$ with $f(\Lambda)\subseteq\Z$.
\end{lemma}

\begin{proof}
If $f\ne0$ and $f(\Lambda)\subseteq\Z$, then $\Lambda\subseteq f^{-1}(\Z)$, a
proper closed set. Conversely, the closure $H$ of $\Lambda$ is a closed
subgroup, so $H=V\oplus L$ with $V$ a linear subspace and $L$ a discrete group
generated by vectors linearly independent modulo $V$ (\cite[Ch.~VII, \S1,
Thm.~2]{Bou}). If $H\ne\R^d$, then either $V\oplus\operatorname{span}L\ne\R^d$,
and a non-zero $f$ vanishing on $V\oplus\operatorname{span}L$ works, or $L\ne0$, and
we may take $f$ vanishing on $V$ and on all but one generator of $L$, and equal
to $1$ on it.
\end{proof}

\begin{lemma}[Criterion for generating pairs]\label{lem:gen}
Let $a,b\in\K^*$ with $|a|<1<|b|$.
\begin{enumerate}[label=\textup{(\alph*)}]
\item $(a,b)$ is generating iff the subgroup $\langle a,b\rangle$ of $\K^*$ is dense
in $\K^*$. In that case, for every $x\in\K$ there are $m_i,k_i\in\N$ with
$\min(m_i,k_i)\to\infty$ and $a^{m_i}b^{k_i}\to x$.
\item For $\K=\R$: $(a,b)$ is generating iff $\ln|a|/\ln|b|\notin\Q$ and at
least one of $a,b$ is negative. For $a,b>0$, $\{a^mb^k\}$ is dense in
$(0,\infty)$ iff $\ln a/\ln b\notin\Q$, and then the conclusion of
\textup{(a)} holds for every $x\in[0,\infty)$.
\item For $\K=\C$, write $a=e^{\alpha+i\theta}$, $b=e^{\beta+i\varphi}$ with
$\alpha,\beta,\theta,\varphi\in\R$. Then $(a,b)$ is generating iff there is no
$(x,y)\in\R^2\setminus\{0\}$ such that
\[
\alpha x+\theta y,\quad \beta x+\varphi y,\quad 2\pi y
\]
are all integers.
\end{enumerate}
\end{lemma}

\begin{proof}
(a) If $\{a^mb^k:m,k\in\N\}$ is dense in $\K$, it is dense in $\K^*$, so the
group is dense. The converse and the last assertion follow from
Lemma~\ref{lem:large} with $G=\K^*$; the case $x=0$ follows from density near $0$.

(b) If $a,b>0$ the orbit lies in $(0,\infty)$; if $\ln|a|/\ln|b|=p/q$ is
rational then $\ln|a^mb^k|$ lies in the discrete group
$\frac{\ln|b|}{q}\Z$. Conversely, suppose $\ln|a|/\ln|b|\notin\Q$ and, say,
$a<0$. The positive part $H_+=\langle a,b\rangle\cap(0,\infty)$ contains $a^2$
and $b^2$, and $2\ln|a|\Z+2\ln|b|\Z$ is dense in $\R$. So $H_+$ is dense in
$(0,\infty)$, and $aH_+$ is dense in $(-\infty,0)$. The case $b<0$ is the same.
The statement for $a,b>0$ follows in the same way from Lemma~\ref{lem:large}
with $G=(0,\infty)$.

(c) The map $\pi\colon\R^2\to\C^*$, $(s,t)\mapsto e^{s+it}$, is a continuous open
surjective homomorphism with kernel $\{0\}\times2\pi\Z$. The subgroup
$\langle a,b\rangle$ is dense iff its preimage
$\Lambda=\Z(\alpha,\theta)+\Z(\beta,\varphi)+\Z(0,2\pi)$ is dense in $\R^2$.
Apply Lemma~\ref{lem:kron} with $f(s,t)=xs+yt$.
\end{proof}

\section{Three lemmas on triangular matrices}\label{sec:lemmas}

\begin{lemma}\label{lem:compare}
Let $a,b,c,d\in\K$ with $|a|,|c|<1<|b|,|d|$ and
$\frac{\ln|a|}{\ln|b|}<\frac{\ln|c|}{\ln|d|}$. If $m_i\to\infty$ and
$\{|c^{m_i}d^{k_i}|\}_i$ is bounded above, then $a^{m_i}b^{k_i}\to0$.
\end{lemma}

\begin{proof}
If $|c^{m_i}d^{k_i}|<M$, then
$k_i<-\frac{\ln|c|}{\ln|d|}m_i+\frac{\ln M}{\ln|d|}$. Since $\ln|b|>0$,
\[
m_i\ln|a|+k_i\ln|b|=\ln|b|\Bigl(\frac{\ln|a|}{\ln|b|}m_i+k_i\Bigr)
< m_i\ln|b|\Bigl(\frac{\ln|a|}{\ln|b|}-\frac{\ln|c|}{\ln|d|}\Bigr)+\frac{\ln M\,\ln|b|}{\ln|d|},
\]
which tends to $-\infty$.
\end{proof}

\begin{lemma}\label{lem:growth}
Let $M$ be an invertible lower triangular $n\times n$ matrix. For $i\ge j$ let
$\mu_{ij}=\max_{j\le k\le i}|M_k|$, and suppose that for every such pair the
maximum is attained at a single index. Then there is $\lambda>0$ such that
$|(M^l)_{ij}|\le\lambda\,\mu_{ij}^{\,l}$ for all $i\ge j$ and $l\ge0$. In
particular, if $A$ is lower triangular with $0<|A_1|<\dots<|A_n|$, there is
$\lambda=\lambda(A)$ with
\begin{equation}\label{eq:growth}
|(A^l)_{ij}|\le\lambda|A_i|^l,\qquad |(A^{-l})_{ij}|\le\lambda|A_j|^{-l}\qquad(l\ge0).
\end{equation}
\end{lemma}

\begin{proof}
Write $M=D+L$ with $D$ diagonal and $L$ strictly lower triangular. Expanding
$M^l$ as a sum of products of $D$'s and $L$'s and following the indices,
\[
(M^l)_{ij}=\sum_{r\ge0}\ \sum_{j=k_0<k_1<\dots<k_r=i} L_{k_rk_{r-1}}\cdots L_{k_1k_0}\;
h_{l-r}(M_{k_0},\dots,M_{k_r}),
\]
where $h_m(x_0,\dots,x_r)=\sum_{e_0+\dots+e_r=m}x_0^{e_0}\cdots x_r^{e_r}$ is the
complete homogeneous symmetric polynomial (and $h_m=0$ for $m<0$). If
$\mu=|x_{s^*}|$ is the strict maximum of $|x_0|,\dots,|x_r|$, then
\[
|h_m(x_0,\dots,x_r)|\le\mu^m\sum_{e}\prod_{s\ne s^*}\Bigl(\frac{|x_s|}{\mu}\Bigr)^{e_s}
\le\mu^m\prod_{s\ne s^*}\Bigl(1-\frac{|x_s|}{\mu}\Bigr)^{-1}.
\]
For a chain from $j$ to $i$ the maximum of $|M_{k_s}|$ is $\mu_{ij}$, by
hypothesis attained once. Summing over the finitely many chains gives the
bound with
\[
\lambda=\max_{i\ge j}\sum_{\text{chains}}|L_{k_rk_{r-1}}\cdots L_{k_1k_0}|\,
\mu_{ij}^{-r}\prod_{s\ne s^*}\bigl(1-|M_{k_s}|/\mu_{ij}\bigr)^{-1}.
\]
For $M=A$ the maximum over $j\le k\le i$ is $|A_i|$. For $M=A^{-1}$, which
is lower triangular with diagonal $1/A_k$ of decreasing modulus, it is
$|A_j|^{-1}$.
\end{proof}

\begin{lemma}\label{lem:J}
Let $A$ be lower triangular with $0<|A_1|<|A_k|$ for all $k\ge2$, and let $u$
be defined by \eqref{eq:u}. Then
\[
\lim_{l\to\infty}(A_1A^{-1})^l=J:=u\,e_1^{T},
\]
that is, $J_{i1}=u_i$ and $J_{ij}=0$ for $j\ne1$. Moreover $u$ is the first
column of $(A_1^{-1}A-I+\Delta)^{-1}$.
\end{lemma}

\begin{proof}
$P=A_1A^{-1}$ is lower triangular with diagonal entries $1$ and $A_1/A_k$
($k\ge2$), the latter of modulus $<1$. So $1$ is a simple eigenvalue of $P$
and all other eigenvalues lie in the open unit disc. The vector $u$ is a right
eigenvector for $1$, since $Pu=u\iff Au=A_1u$. The vector $e_1^T$ is a left
eigenvector, since the first row of $P$ is $e_1^T$. By the Jordan form,
$P^l\to ue_1^T/(e_1^Tu)=ue_1^T$.

For the last claim put $X=A_1^{-1}A-I+\Delta$. It is lower triangular with
diagonal $1,A_2/A_1-1,\dots,A_n/A_1-1$, hence invertible, and
$Xu=(A_1^{-1}Au-u)+u_1e_1=e_1$. Thus $X^{-1}e_1=u$.
\end{proof}

\begin{lemma}[Slopes]\label{lem:slopes}
Assume \eqref{eq:diag} and \eqref{eq:chain}. Put $P_0=(0,0)$ and
$P_i=(\ln|A_i|,\ln|B_i|)\in\R^2$, and for $0\le m<q\le n$ let $\sigma(q,m)$ be
the slope of the segment $P_mP_q$. Then:
\begin{enumerate}[label=\textup{(\alph*)}]
\item $\sigma(q,0)<\sigma(1,0)$ for $2\le q\le n$;
\item $\sigma(q,m)<\sigma(s+1,1)$ whenever $1\le m\le s<q\le n$ and $(q,m)\ne(s+1,1)$.
\end{enumerate}
\end{lemma}

\begin{proof}
By \eqref{eq:diag} the first coordinates of $P_0,P_1,\dots,P_n$ increase
strictly. Condition \eqref{eq:chain} says
$\sigma(n,1)<\dots<\sigma(2,1)<\sigma(1,0)$. We use the elementary fact that
for points $X,Y,Z$ with increasing first coordinates, the slope of $XZ$ is a
convex combination, with positive weights, of the slopes of $XY$ and $YZ$.

(a) Apply it to $P_0,P_1,P_q$: since $\sigma(q,1)<\sigma(1,0)$, also
$\sigma(q,0)<\sigma(1,0)$.

(b) If $m=1$ then $q\ge s+2$, and $\sigma(q,1)<\sigma(s+1,1)$ by
\eqref{eq:chain}. If $m\ge2$, apply the fact to $P_1,P_m,P_q$: since
$\sigma(q,1)<\sigma(m,1)$, we get $\sigma(q,m)<\sigma(q,1)\le\sigma(s+1,1)$.
\end{proof}

Note that $\sigma(q,m)=\ln|B_q/B_m|\big/\ln|A_q/A_m|$. In particular (a) says
$(B_q,A_q)$ satisfies the slope inequality of Lemma~\ref{lem:compare} relative
to $(B_1,A_1)$.

\section{Proof of Theorem~\ref{thm:main}}\label{sec:proof}

Fix $p$ with $p_1\ne0$ and $N\in\N$. For $1\le s\le n$ let
\[
W_s=\bigl\{B^{\delta_s}A^{\gamma_s}\cdots B^{\delta_1}A^{\gamma_1}p:\ \gamma_i,\delta_i\ge N\bigr\}.
\]
We prove by induction on $s$ that
\begin{equation}\label{eq:ind}
\K^s\times\{0\}^{n-s}\subseteq\cl(W_s).
\end{equation}
For $s=n$ this is the theorem.

\emph{Base case $s=1$.} Let $x_1\in\K$. Since $(B_1,A_1)$ is generating,
Lemma~\ref{lem:gen}(a) gives $k_i,l_i\ge\max(N,i)$ with
$B_1^{k_i}A_1^{l_i}\to x_1/p_1$. Then $(B^{k_i}A^{l_i}p)_1=B_1^{k_i}A_1^{l_i}p_1\to x_1$.
For $j\ge2$, \eqref{eq:growth} gives
\[
|(B^{k}A^{l}p)_j|\le\lambda|B_j|^{k}|A_j|^{l}\sum_t|p_t|.
\]
By Lemma~\ref{lem:slopes}(a) and Lemma~\ref{lem:compare}, applied with
$(a,b)=(B_j,A_j)$ and $(c,d)=(B_1,A_1)$, the right side tends to $0$ along
$(k_i,l_i)$. Hence $(x_1,0,\dots,0)^T\in\cl(W_1)$.

\emph{Inductive step.} Assume \eqref{eq:ind} for some $s<n$. Write
\[
A=\begin{pmatrix}S&0\\ Y&T\end{pmatrix},\qquad B=\begin{pmatrix}U&0\\0&V\end{pmatrix},
\]
with $S,U$ of size $s\times s$. For $k,l\in\N$ define the $(n-s)\times s$
matrix $O_{k,l}$ by
\begin{equation}\label{eq:O}
B^kA^l\begin{pmatrix}I_s\\0\end{pmatrix}=\begin{pmatrix}I_s\\ O_{k,l}\end{pmatrix}U^kS^l .
\end{equation}
Since $A^{-l}$ is lower triangular with upper-left block $S^{-l}$, for
$1\le j\le n-s$ and $1\le m\le s$ we have
\begin{equation}\label{eq:Ojm}
(O_{k,l})_{jm}=\Bigl(\frac{B_{s+j}}{B_m}\Bigr)^k\sum_{t=1}^{s}(A^l)_{s+j,t}\,(A^{-l})_{tm}.
\end{equation}

\emph{The $(1,1)$ entry.} Since $\sum_{t=1}^n(A^l)_{s+1,t}(A^{-l})_{t1}=0$ and
$(A^l)_{s+1,t}=0$ for $t>s+1$,
\[
\sum_{t=1}^{s}(A^l)_{s+1,t}(A^{-l})_{t1}=-A_{s+1}^{\,l}(A^{-l})_{s+1,1}.
\]
Therefore
\begin{equation}\label{eq:O11}
(O_{k,l})_{11}=-\Bigl(\frac{B_{s+1}}{B_1}\Bigr)^k\Bigl(\frac{A_{s+1}}{A_1}\Bigr)^l
\bigl((A_1A^{-1})^l\bigr)_{s+1,1}.
\end{equation}
By Lemma~\ref{lem:J}, $((A_1A^{-1})^l)_{s+1,1}\to\omega:=u_{s+1}$, and
$\omega\ne0$ by (ii).

Let $\alpha\in\K$. Since $(B_{s+1}/B_1,A_{s+1}/A_1)$ is generating,
Lemma~\ref{lem:gen}(a) gives $k_i,l_i\ge\max(N,i)$ with
\begin{equation}\label{eq:target}
\Bigl(\frac{B_{s+1}}{B_1}\Bigr)^{k_i}\Bigl(\frac{A_{s+1}}{A_1}\Bigr)^{l_i}\to-\frac{\alpha}{\omega},
\end{equation}
and then $(O_{k_i,l_i})_{11}\to\alpha$ by \eqref{eq:O11}.

\emph{The other entries.} By \eqref{eq:Ojm} and \eqref{eq:growth},
\[
|(O_{k,l})_{jm}|\le s\lambda^2\Bigl|\frac{B_{s+j}}{B_m}\Bigr|^k\Bigl|\frac{A_{s+j}}{A_m}\Bigr|^l .
\]
Here $|B_{s+j}/B_m|<1<|A_{s+j}/A_m|$ since $s+j>m$. If $(j,m)\ne(1,1)$,
Lemma~\ref{lem:slopes}(b) (with $q=s+j$) gives
$\sigma(s+j,m)<\sigma(s+1,1)$. The left side of \eqref{eq:target} is bounded,
so Lemma~\ref{lem:compare} with $(a,b)=(B_{s+j}/B_m,A_{s+j}/A_m)$ and
$(c,d)=(B_{s+1}/B_1,A_{s+1}/A_1)$ shows $(O_{k_i,l_i})_{jm}\to0$. Hence
\[
O_{k_i,l_i}\to\alpha E,\qquad E=e_1e_1^T\in\K^{(n-s)\times s}.
\]

\emph{Conclusion.} Let $y\in\K^{s+1}$ with $y_1\ne0$, put
$\alpha=y_{s+1}/y_1$ and $y'=(y_1,\dots,y_s)^T$, choose $(k_i,l_i)$ as above,
and set $x^{(i)}=S^{-l_i}U^{-k_i}y'$. By \eqref{eq:O},
\[
B^{k_i}A^{l_i}\begin{pmatrix}x^{(i)}\\0\end{pmatrix}
=\begin{pmatrix}y'\\ O_{k_i,l_i}y'\end{pmatrix}\longrightarrow
\begin{pmatrix}y'\\ \alpha y_1e_1\end{pmatrix}=\begin{pmatrix}y\\0\end{pmatrix}.
\]
By the induction hypothesis $(x^{(i)},0)\in\cl(W_s)$, so by continuity
$B^{k_i}A^{l_i}(x^{(i)},0)\in\cl(B^{k_i}A^{l_i}W_s)\subseteq\cl(W_{s+1})$, using
$k_i,l_i\ge N$. Hence $(y,0)\in\cl(W_{s+1})$ for all $y$ with $y_1\ne0$, and
by taking closures for all $y\in\K^{s+1}$. This proves \eqref{eq:ind} for $s+1$.
\qed

\begin{remark}
The proof uses (ii) only through $u_{s+1}\ne0$ for each $s$, and
\eqref{eq:chain} only through Lemma~\ref{lem:slopes} and the generating
property of the $n$ pairs involved. It is also invariant under simultaneous
conjugation: if $PAP^{-1}$ and $PBP^{-1}$ satisfy the hypotheses, then every
$x$ with $(Px)_1\ne0$ has a dense orbit under $\langle A,B\rangle$.
\end{remark}

\section{Explicit examples}\label{sec:examples}

\begin{proposition}\label{prop:examples}
Let $A$ be the $n\times n$ lower triangular matrix with
\[
A_{kk}=e^{k}\ (1\le k\le n),\qquad A_{k1}=e\ (2\le k\le n),\qquad A_{kl}=0\ \text{otherwise}.
\]
\begin{enumerate}[label=\textup{(\alph*)}]
\item \textup{(Real case)} Let $B=\operatorname{diag}(B_1,\dots,B_n)$ with
$B_1=-e^{-\sqrt2}$ and $B_k=e^{-\sqrt2\,k^2}$ for $k\ge2$.
\item \textup{(Complex case)} Let $B=\operatorname{diag}(B_1,\dots,B_n)$ with
$B_k=e^{(i-\sqrt2)k^2}$ for $k\ge1$.
\end{enumerate}
In both cases $(A,B)$ satisfies the hypotheses of Theorem~\ref{thm:main}. Hence
every $p\in\K^n$ with $p_1\ne0$ has a dense orbit under $\langle A,B\rangle$.
\end{proposition}

\begin{proof}
Condition \eqref{eq:diag} is clear: $|B_k|=e^{-\sqrt2k^2}$ and $|A_k|=e^k$. For
(ii), $u_k=-A_{k1}/(A_k-A_1)=-1/(e^{k-1}-1)\ne0$ for $k\ge2$. The slopes are
\[
\frac{\ln|B_1|}{\ln|A_1|}=-\sqrt2,\qquad
\frac{\ln|B_k/B_1|}{\ln|A_k/A_1|}=\frac{-\sqrt2(k^2-1)}{k-1}=-\sqrt2(k+1)\quad(k\ge2),
\]
which are strictly decreasing in $k$. It remains to check that the pairs
$(B_1,A_1)$ and $(B_k/B_1,A_k/A_1)$, $2\le k\le n$, are generating.

(a) The pair $(B_1,A_1)=(-e^{-\sqrt2},e)$ has $\ln$-ratio $-\sqrt2\notin\Q$ and
a negative entry. The pair
$(B_k/B_1,A_k/A_1)=(-e^{-\sqrt2(k^2-1)},e^{k-1})$ has ratio
$-\sqrt2(k+1)\notin\Q$ and a negative entry. Apply Lemma~\ref{lem:gen}(b).

(b) Put $c=k^2-1$ and $d=k-1$ for $k\ge2$, and $c=d=1$ for the pair
$(B_1,A_1)$. In all cases the pair is $(e^{c(i-\sqrt2)},e^{d})$. By
Lemma~\ref{lem:gen}(c) we must exclude $(x,y)\ne0$ with
\[
-\sqrt2\,c\,x+c\,y\in\Z,\qquad d\,x\in\Z,\qquad 2\pi y\in\Z .
\]
Write $dx=l$ and $2\pi y=j$, and let $N=c(-\sqrt2\,l/d+j/2\pi)\in\Z$. If $j\ne0$,
then $N+\sqrt2\,cl/d\ne0$ and
$\pi=cj\big/\bigl(2(N+\sqrt2\,cl/d)\bigr)$ is algebraic, contradicting
Lindemann's theorem. If $j=0$, then $\sqrt2\,cl/d=-N\in\Q$ forces $l=0$. So
$x=y=0$.
\end{proof}

\section{Affine maps and somewhere dense orbits}\label{sec:apps}

\begin{lemma}\label{lem:exists}
For every $\rho<0$ there is a generating pair $(b,a)$ over $\K$ with
$\ln|b|/\ln|a|\in(\rho,0)$.
\end{lemma}

\begin{proof}
Take $a=e$. For $\K=\R$ let $b=-e^{-t}$ with $t\in(0,-\rho)$ irrational, and
apply Lemma~\ref{lem:gen}(b). For $\K=\C$ let $b=e^{-t+i}$ with $t\in(0,-\rho)$
algebraic and irrational (e.g.\ $t=\sqrt2/m$ for large $m$). If $x\in\Z$,
$2\pi y=j\in\Z$ and $-tx+j/2\pi\in\Z$, then $j\ne0$ would make $\pi$
algebraic, while $j=0$ forces $x=0$ because $t\notin\Q$. Apply
Lemma~\ref{lem:gen}(c).
\end{proof}

\begin{proof}[Proof of Theorem~\ref{thm:affine}]
By Lemma~\ref{lem:exists} choose a generating pair $(b,a)$ with
$(B_1,A_1)\prec(b,a)$; note $|b|<1<|a|$. Define the $(n+1)\times(n+1)$ matrices
\[
A'=\begin{pmatrix}a&0\\ av&aA\end{pmatrix},\qquad B'=\begin{pmatrix}b&0\\0&bB\end{pmatrix}.
\]
Then $A'$ is lower triangular and $B'$ diagonal, with $A'_1=a$,
$A'_{i+1}=aA_i$, $B'_1=b$ and $B'_{i+1}=bB_i$. Condition \eqref{eq:diag} for
$(A',B')$ is immediate. Since $(B'_{i+1}/B'_1,A'_{i+1}/A'_1)=(B_i,A_i)$,
condition \eqref{eq:chain} for $(A',B')$ is exactly
$(B_n,A_n)\prec\dots\prec(B_1,A_1)\prec(b,a)$. For (ii), the eigenvector of
$A'$ for $a$ with first entry $1$ is $u'=(1,w)$, where $av+aAw=aw$, i.e.\
$w=-(A-I)^{-1}v$. All its entries are non-zero by hypothesis.

By Theorem~\ref{thm:main}, $\langle A',B'\rangle(1,p)$ is dense in $\K^{n+1}$ for
every $p\in\K^n$. Let $\Phi\colon\K^*\times\K^n\to\K^n$,
$\Phi(y_0,y)=y/y_0$, and $\Psi(x)=(1,x)$. Then $\Phi(A'\Psi x)=Ax+v=D(x)$ and
$\Phi(B'\Psi x)=Bx=E(x)$. Since $\Phi(Ly)=\Phi(L\Psi\Phi y)$ for linear $L$ and
$y_0\ne0$, the map $L\mapsto\Phi\circ L\circ\Psi$ is a semigroup homomorphism
of $\langle A',B'\rangle$ onto $\langle D,E\rangle$. Hence
$\langle D,E\rangle p=\Phi(\langle A',B'\rangle\Psi p)$. The set
$\langle A',B'\rangle\Psi p$ lies in $\K^*\times\K^n$, where the first coordinate
is $a^{\gamma}b^{\delta}\ne0$, and it is dense there. Since $\Phi$ is a
continuous surjection, $\langle D,E\rangle p$ is dense in $\K^n$.
\end{proof}

\begin{proof}[Proof of Theorem~\ref{thm:halfspace}]
Since $A_1,B_1>0$, every element of the orbit has first coordinate of the form
$B_1^{\delta}A_1^{\gamma}p_1>0$, so the closure lies in $\{x_1\ge0\}$.
For the reverse inclusion, repeat the proof of Theorem~\ref{thm:main} with
\eqref{eq:ind} replaced by
\[
\{x\in\R^s\times\{0\}^{n-s}:x_1>0\}\subseteq\cl(W_s).
\]
Lemma~\ref{lem:slopes} uses only the slopes, so it applies by (b). In the base
case the target $x_1/p_1$ is positive, and Lemma~\ref{lem:gen}(b) for the
positive pair $(B_1,A_1)$ provides $k_i,l_i\to\infty$ with
$B_1^{k_i}A_1^{l_i}\to x_1/p_1$. In the inductive step take $y$ with $y_1>0$.
The pair $(B_{s+1}/B_1,A_{s+1}/A_1)$ is generating in $\R$, so every
$\alpha\in\R$ is attained. The first coordinate of
$x^{(i)}=S^{-l_i}U^{-k_i}y'$ is $A_1^{-l_i}B_1^{-k_i}y_1>0$, so
$(x^{(i)},0)$ lies in the set covered by the induction hypothesis. The rest is
unchanged.
\end{proof}

\begin{example}\label{ex:halfspace}
Let $A$ be as in Proposition~\ref{prop:examples}, and let $B_1=e^{-\sqrt2}$ and
$B_k=-e^{-\sqrt2k^2}$ for $k\ge2$. Then $\ln B_1/\ln A_1=-\sqrt2\notin\Q$, and
$(B_k/B_1,A_k/A_1)=(-e^{-\sqrt2(k^2-1)},e^{k-1})$ is generating with slope
$-\sqrt2(k+1)$. By Theorem~\ref{thm:halfspace}, every orbit is contained in one
of $\{x_1>0\}$, $\{x_1<0\}$, $\{x_1=0\}$. Those with $p_1\ne0$ have closure
equal to a closed half-space. No orbit is dense, yet every orbit off the
hyperplane $\{x_1=0\}$ is somewhere dense. This works for every $n\ge1$.
\end{example}

\section{Further remarks and open problems}\label{sec:problems}

\subsection*{Somewhere dense orbits}
Feldman~\cite{Fel} constructed a $2n$-tuple on $\R^n$ with a somewhere dense
orbit that is not dense, and proved that no finite tuple on $\C^n$ has this
property. Example~\ref{ex:halfspace} shows that, without commutativity, two
generators suffice on $\R^n$ for every $n$. The complex case remains open.

\begin{problem}\label{prob:sd}
Is there, for some $n\ge1$, a pair of complex $n\times n$ matrices with an
orbit that is somewhere dense but not dense in $\C^n$?
\end{problem}

\noindent The construction of Theorem~\ref{thm:halfspace} has no direct analogue
over $\C$. In the complex case, the closure of a two-generator subsemigroup of
$\C^*$ with non-empty interior contains an open subgroup of the connected group
$\C^*$, hence is everything. So the first coordinate cannot be confined to a
region with non-empty interior in this way.

\subsection*{Orbits of matrices}
\begin{problem}\label{prob:mat}
Are there $n\times n$ matrices $A,B$ and an $n\times k$ matrix $C$ over $\K$
such that $\langle A,B\rangle C$ is dense in $\K^{n\times k}$?
\end{problem}

\noindent For $k=1$ this is Theorem~\ref{thm:main}. The following proposition
limits the possible examples.

\begin{proposition}\label{prop:k}
Let $\Gamma=\langle A,B\rangle$ and $C\in\K^{n\times k}$.
\begin{enumerate}[label=\textup{(\alph*)}]
\item If $\operatorname{rank}C<k$ (in particular if $k>n$), then $\Gamma C$ is not dense.
\item If $k\ge2$ and $A^T,B^T$ have a common eigenvector (for example, if $A$ and $B$ are both lower triangular), then $\Gamma C$ is not dense.
\item If $k=n$, then $\Gamma C$ is dense iff $C$ is invertible and $\Gamma$ is dense in $\K^{n\times n}$.
\end{enumerate}
\end{proposition}

\begin{proof}
(a) Every row of $TC$ lies in the row space of $C$. So $\Gamma C$ lies in the
proper linear subspace of matrices whose rows lie in that row space.

(b) If $w^TA=aw^T$ and $w^TB=bw^T$ with $w\ne0$, then $w^T(TC)$ is a scalar
multiple of $w^TC$ for every $T\in\Gamma$. But $M\mapsto w^TM$ is a linear
surjection $\K^{n\times k}\to\K^{1\times k}$, which maps dense sets to dense
sets. A line is not dense in $\K^{1\times k}$ for $k\ge2$. For lower triangular
matrices take $w=e_1$.

(c) By (a), $C$ must be invertible. Right multiplication by $C$ is then a
homeomorphism of $\K^{n\times n}$.
\end{proof}

\noindent Thus Problem~\ref{prob:mat} is open only for $2\le k\le n$, for pairs
with no common left eigenvector. For $k=n$ it asks for a dense two-generator
subsemigroup of $\K^{n\times n}$.

\subsection*{Multi-hypercyclicity}
Herrero~\cite{Her} conjectured, and Costakis~\cite{Cos} and Peris~\cite{Per}
proved, that a multi-hypercyclic operator is hypercyclic.

\begin{problem}\label{prob:multi}
Let $A,B$ be $n\times n$ matrices over $\K$ and suppose
$\bigcup_{j=1}^k\langle A,B\rangle x^{(j)}$ is dense in $\K^n$. Must some
$\langle A,B\rangle x^{(j)}$ be dense?
\end{problem}

\noindent Corollary~\ref{cor:main}(c) gives an affirmative answer for all pairs
satisfying the hypotheses of Theorem~\ref{thm:main}.

\subsection*{Acknowledgement}
I would like to thank George Costakis for reading the original manuscript and
for his useful comments, and the referees of an earlier version for many
corrections. This revised version is prepared by an AI assistant (Claude).

\end{document}